\documentclass[11pt]{amsart}
\usepackage{amssymb,amsmath,amsfonts,amscd,euscript}
\usepackage[all]{xy}
\usepackage{color}
\usepackage{tikz}
\usepackage{tikz-cd}

\usepackage[backref=page]{hyperref}
\usepackage{todonotes}

\usepackage{amsthm,enumitem,bm,pict2e,xcolor}
\usepackage{indentfirst}

\newcommand{\nc}{\newcommand}

\numberwithin{equation}{section}
\newtheorem{thm}{Theorem}[section]
\newtheorem{prop}[thm]{Proposition}
\newtheorem{lem}[thm]{Lemma}
\newtheorem{cor}[thm]{Corollary}
\newtheorem{rem}[thm]{Remark}

\newtheorem{example}[thm]{Example}

\newtheorem{theoremA}{Theorem}

\nc{\gl}{\mathfrak{gl}}
\nc{\msl}{\mathfrak{sl}}
\nc{\fh}{\mathfrak{h}}
\nc{\GL}{\mathfrak{GL}}
\nc{\g}{\mathfrak{g}}
\nc{\gh}{\widehat\g}
\nc{\la}{\lambda}
\nc{\al}{\alpha }
\nc{\be}{\beta }
\nc{\ve}{\varepsilon }
\nc{\om}{\omega }

\nc{\ch}{{\mathop {\rm ch}}}
\nc{\Tr}{{\mathop {\rm Tr}\,}}
\nc{\Id}{{\mathop {\rm Id}}}
\nc{\bra}{\langle}
\nc{\ket}{\rangle}
\nc{\bs}{{\bf s}}
\nc{\bbm}{{\bf m}}
\nc{\bp}{{\bf p}}
\nc{\bt}{{\bf t}}
\nc{\pa}{\partial}
\nc{\ld}{\ldots}
\nc{\cd}{\cdots}
\nc{\hk}{\hookrightarrow}
\nc{\T}{\otimes}
\nc{\Gr}{\mathrm{Gr}}
\nc{\ov}{\overline}

\nc{\cO}{\mathcal O}
\nc{\cF}{\mathcal F}
\nc{\cL}{\mathcal L}
\nc{\mgl}{\mathfrak{gl}}
\nc{\U}{\mathrm U}
\nc{\V}{\EuScript V}
\nc{\bH}{\EuScript H}
\nc{\Res}{\mathrm{Res\ }}

\newcommand{\bC}{{\mathbb C}}
\newcommand{\bR}{{\mathbb R}}
\newcommand{\bZ}{{\mathbb Z}}

\newcommand{\bP}{{\mathbb P}}
\newcommand{\bA}{{\mathbb A}}
\newcommand{\bG}{{\mathbb G}}
\newcommand{\fp}{{\mathfrak p}}
\newcommand{\fg}{{\mathfrak g}}
\newcommand{\fb}{{\mathfrak b}}

\newcommand{\G}{{\mathfrak G}}

\newcommand{\fn}{{\mathfrak n}}
\newcommand{\fr}{{\mathfrak r}}

\newcommand{\Fl}{\EuScript{F}}

\newcommand{\Hom}{\mathrm{Hom}}
\newcommand{\pr}{\mathrm{pr}}

\newcommand{\bd}{{\bf d}}

\begin{document}

\title[Birational maps, PBW degenerate flags  and poset polytopes]
{Birational maps,  PBW degenerate flags  and poset polytopes}

\author{Evgeny Feigin}
\address{Evgeny Feigin:\newline
School of Mathematical Sciences, Tel Aviv University, Tel Aviv, 69978, Israel}
\email{evgfeig@gmail.com}

\begin{abstract}
We extend the results on the graph closures of the birational maps between projective
spaces and Grassmannians to the
case of PBW degenerate flag varieties. 
The advantage of the PBW degenerate flags (as opposed to their classical
analogues) is the existence of a large group of symmetries for the graph closures. 
We discuss the combinatorial, algebraic and geometric sides of the picture. In particular, we show that toric 
degenerations of Borovik, Sturmfels and Sverrisd\'ottir are still available
in the general settings. We also derive a description of the graph closures
for flag varieties 
in terms of quiver representations. 
\end{abstract}

\maketitle

\section*{Introduction}
In \cite{BSS,F5} the closures of the graphs of birational maps 
from projective spaces to Grassmannians were studied from combinatorial, algebraic and
geometric sides (see also \cite{FO,FSS}). 
The goal of this paper is to extend the whole picture  to the case 
of arbitrary type $A$ flag varieties. A brief outcome is that such an extension does
exist, but our approach works well not for the classical partial flag varieties,
but for their PBW degenerations \cite{F1,F2}. Let us describe our results in
more detail.

The general geometric construction we deal with is as follows.
Let $X$ be a projective algebraic variety admitting an open cell 
$\imath: \bA^N\to X$. Treating  the affine space $\bA^N$ as an open part 
of the projective space $\bP^N$ one gets a birational extension 
$\imath:\bP^N\to X$. Then one wants to study the closure of the graph 
$\{(x,\imath(x)), x\in\bA^N\}$ inside the product $\bP^N\times X$ (see \cite{KP,BSS,F5} for
special cases of this construction).

In \cite{F5} we studied the case $X=\Gr(d,n)$ -- the Grassmann variety of $d$-dimensional subspaces in an $n$-dimensional vector space. It is natural
to ask if one can generalize the available results to the case of flag 
varieties. Unfortunately, the approach of \cite{F5} does not apply to the case
of classical flag varieties, because the group of symmetries available
in the Grassmann case gets broken. However, if one replaces the classical flag varieties with their PBW degeneration \cite{F1,F2,F4}, then the group 
of symmetries survives (this group is certain degeneration $SL_n^a$ of the
classical special linear group \cite{BNV,BR,F3,KZ-J,PY1}). Hence in this paper we 
deal with the PBW degenerate flag varieties.

Recall that the full $SL_n$ PBW degenerate flag variety $\Fl_n^a$ can be defined in
pure Lie theoretic terms, but also admits an explicit description in terms of
sequences of subspaces of an $n$-dimensional vector space. One can reformulate
this description in a more conceptual way. Namely, let us consider the equioriented  type $A_{n-1}$ quiver $Q$. Then one identifies $\Fl_n^a$ with
a quiver Grassmannian $\Gr_{\dim P}(P\oplus I)$, where $P$ and $I$ are certain
projective and injective $Q$-modules (see \cite{CFR1,CFR2}).

The degenerate flag varieties admit an open cell and we consider the graph
closure $\G(n)\subset \bP^N\times \Fl_n^a$, where $N=n(n-1)/2$ is the dimension
of the flag variety. We note that by definition one gets a projection 
$\varphi:\G(n)\to \Fl_n^a$.
An important observation is that the action of the 
degenerate Lie group $SL_n^a$ extends from $\Fl_n^a$ to $\G(n)$ and $\varphi$ is equivariant 
with respect to this action. This observation allows to prove the following theorem.

\begin{theoremA}
The fiber $\varphi^{-1}(U)$ over a point $U\in \Fl_n^a$ is isomorphic 
to the projectivization
of the space $\Hom_Q(P/N_P,N_I)$, where $U\simeq N_P\oplus N_I$ as a point of the
quiver Grassmannian $\Gr_{\dim P}(P\oplus I)$, $N_P\subset P$, $N_I\subset I$.
\end{theoremA}	
   
As in the classical case, the study of the geometry of the flag varieties 
comes together with the study of certain combinatorial and algebraic structures.
In \cite{F5} we introduced certain generalization of the classical 
poset polytopes (see \cite{HOT,Stan,T}). In order to adjust the setup for the
flagged case, we introduce the marked version of the polytopes in the 
spirit of \cite{ABS,FaFo}. The resulting polytopes $X_{\la,M}$ depend on
a dominant integral $\msl_n$ weight $\la$ and on a non-negative integer $M$. 
Let $S_{\la,M}\subset X_{\la,M}$ be the set of integral points. We prove the
following theorem, which is important for understanding of the basic algebraic and geometric
properties of the graph closures $\G(n)$.      

\begin{theoremA}
For any dominant integral weights $\la,\la'$ and $M,M'\ge 0$ one has
$S_{\la,M}+S_{\la',M'}=S_{\la+\la',M+M'}$ and $X_{\la,M}+X_{\la',M'}=X_{\la+\la',M+M'}$.
\end{theoremA}	
  
Recall that the classical flag varieties admit embeddings to the projectivizations of the irreducible highest weight modules $L_\la$ (see e.g. \cite{Fu}). Similarly,
the PBW degenerate flag varieties can be realized inside the PBW degenerate representation $L_\la^a$ \cite{F1}. We show that varieties $\G(n)$ also admit
similar embeddings into $\bP(L_{\la,M})$ for certain cyclic $SL_n^a$ modules
$L_{\la,M}$, which we construct explicitly. 
Combining all the available
structures together, we obtain several results about representations
$L_{\la,M}$ and varieties $\G(n)$. In particular, one derives the existence 
of certain toric degenerations (in the spirit of \cite{FaFL1,FaFL2,FFL2,M}).

\begin{theoremA}
The elements $f^\bs\ell_{\la,M}$ form a basis of $L^a_{\la,M}$. The homogeneous
coordinate ring of $\G(n)$ is isomorphic to the direct sum of duals of $L^a_{\la,M}$. The varieties $\G(n)$ admit flat degenerations to the  toric varieties
defined by the polytopes $X_{\la,M}$. 
\end{theoremA}

We note that (certain modifications of) all the theorems above hold true for
arbitrary partial degenerate flag varieties. We discuss the general case in a separate (last) section of the paper and restrict to the case of complete flags otherwise.

Our paper is organized as follows.
In Section \ref{sec:comb} we define a family of convex polytopes and prove the Minkowski sum property for the sets of integral points inside these polytopes.   
In Section \ref{sec:repth} we define a class of cyclic representations of the
degenerate special linear Lie algebra, describe the defining relations of these representations and construct monomial bases in terms of the polytopes from Section \ref{sec:comb}. Section \ref{sec:geom} is devoted to the study of
the geometry of the graph closures of the birational maps from a projective
space to the PBW degenerate flag varieties. In particular, we describe the fibers
of the natural projections to the space of flags in terms of the representation
theory of quivers. In Section \ref{sec:partial} we generalize all the results
of the paper to the case of parabolic subalgebras and partial flag varieties.  

\section*{Acknowledgments}
We are grateful to Igor Makhlin and Wojciech Samotij for useful discussions.

\section{Combinatorics}\label{sec:comb}
The goal of this section is to define and study certain (marked) poset polytopes
attached to the poset of positive roots of type $A$. Throughout the paper
we use the notation $[n]=\{1,\dots,n\}$.
  
\subsection{Polytopes}
For a positive integer $n$ we consider the poset $\overline{P}$ consisting of pairs $(i,j)$,
$1\le j \le i\le n$ such that
$(i,j)\le (i_1,j_1)$ if $i\le i_1$ and $j\le j_1$.
Let $P\subset \overline{P}$ be the subposet consisting of pairs $(i,j)$ with $i>j$;
hence $\overline{P}\setminus P =\{(i,i), i\in [n]\}$.
Let $P_d\subset P$, $d\in [n-1]$ be the subposet consisting of pairs $(i,j)$ 
such that 
$i> d\ge j$. We note that $|P|=N=n(n-1)/2$ and $|P_d|=d(n-d)$.

We recall the definition of the FFLV polytope attached to the poset $P$ (see \cite{FFL1,Vin}).
For a collection ${\bbm}=(m_1,\dots,m_{n-1})\in\bZ_{\ge 0}^{n-1}$ let $X_\bbm\subset \bR_{\ge 0}^P$ be the polytope
consisting of points $(x_{i,j})_{(i,j)\in P}$ subject to the following set of inequalities. The 
inequalities are labeled by chains $C\subset P$. Let $(i_{min},j_{min})$ and $(i_{max},j_{max})$
be the minimal and the maximal elements in $C$ (in particular, $j_{min}<i_{max}$). 
Then the inequality corresponding to $C$ is of the form  
\[
\sum_{(i,j)\in C} x_{i,j}\le m_{j_{min}}+m_{j_{min}+1}+\dots + m_{i_{max}-1}.
\]
\begin{rem}
In what follows we denote the number $j_{min}$ by $s(C)$ and the number $i_{max}$
by $f(C)$ (the start and finish of $C$). We note that the chain $C$ can be extended
to a chain $\overline{C}\subset \overline{P}$ by adding two elements $(s(C),s(C))$ and
$(f(C),f(C))$. If one attaches the number $m_1+\dots + m_{d-1}$ to the vertex $(d,d)$, then
$m_{j_{min}}+\dots + m_{i_{max}-1}$ is the difference of numbers attached to 
$(f(C),f(C))$ and $(s(C),s(C))$. 
\end{rem}

We note that $X_\bbm$ are marked chain polytopes (see \cite{ABS,Stan})
with the marked elements from $\overline{P}\setminus P$.
The special poset elements $(d,d)$ are marked by $m_1+\dots+m_{d-1}$ (the 
minimal element $(1,1)$ is marked by $0$). 

Let $S_\bbm= X_\bbm\cap \bZ^P$ be the set of integer points. Then the following holds true \cite{FFL1}: for any
two collections $\bbm,\bbm'\in\bZ_{\ge 0}^{n-1}$ one has  
\[
X_\bbm + X_{\bbm'}=X_{\bbm+\bbm'},\ S_\bbm + S_{\bbm'}=S_{\bbm+\bbm'}.
\]

Now let us fix a non-negative integer $M$. 
We define a polytope $X_{\bbm,M}\subset \bR_{\ge 0}^P$ by the following set of inequalities
labeled by subposets $P'\subset P$:
\begin{equation}
	\sum_{\al\in P'} s_\al \le M + \sum_{d=1}^{n-1} m_d w(P'\cap P_d),
\end{equation} 
where $w(P'\cap P_d)$ is the width of the intersection of $P'$ and $P_d$
\cite{D,Stan}. 
Let $S_{\bbm,M}=X_{\bbm,M}\cap \bZ_{\ge 0}^P$. The main goal of this section is to show that
$S_{\bbm,M} + S_{\bbm',M'}=S_{\bbm+\bbm',M+M'}$.

\begin{example}
Let $\bbm=0$. Then 	$X_{\bbm,M}$ is a simplex consisting of all points with non negative coordinates such that the sum of these coordinates does not exceed $M$. 
\end{example}

\begin{example}
Let $M=0$. Then $X_{\bbm,0}$ coincides with the FFLV polytope $X_\bbm$, which is 
implied by Lemma \ref{lem:manychains} below.
\end{example}

\begin{lem}\label{lem:manychains}
The polytopes $X_{\bbm,M}\subset \bR_{\ge 0}^P$ can be equivalently defined by the the following set of inequalities labeled by collections $C_1,\dots,C_r$ of chains in $P$.
The inequality corresponding to a collection $(C_a)_{a=1}^r$ is given by
 \begin{equation}\label{eq:chainsineq}
\sum_{(i,j)\in \cup C_a} x_{i,j}\le M + \sum_{a=1}^r (m_{s(C_a)}+\dots + m_{f(C_a)-1}).
 \end{equation}
\end{lem}
\begin{proof}
Assume that for any subposet $P'\subset P$ one has
\begin{equation}\label{eq:posetineq}
\sum_{(i,j)\in P'} x_{i,j}\le \sum_{d=1}^{n-1} m_d w(P'\cap P_d).
\end{equation}
Given a collection of chains $C_1,\dots,C_r$ in $P$, let $P'$ be their union.
Then we know that
 \[
\sum_{(i,j)\in P'} x_{i,j}\le M + \sum_{d=1}^{n-1} m_d w(P_d\cap P').
\]
Now the right hand side is no larger than the right hand side of \eqref{eq:chainsineq},
because a chain $C_a$ contributes a summand $m_d$ to the right hand side 
of \eqref{eq:chainsineq} whenever it intersects with $P_d$.
	
We are left to show that if a point $(x_{i,j})\in\bR_{\ge0}^P$ satisfies conditions \eqref{eq:chainsineq}, then for any subposet $P'\subset P$ the inequality
\eqref{eq:posetineq} holds true. Without loss of generality we assume that $x_{i,j}>0$
for $(i,j)\in P'$.
For a subposet $P'$ we construct a specific covering of $P'$ by chains $C_1,\dots,C_r$.
We first define $C_1$. Let $(i_1,j_1)$ be the element if $P'$ with the minimal possible value $j_1$ of the second coordinate and the minimal possible $i_1$ among the 
elements with this property. 
Let $(i_2,j_1)$,
$i_2\ge i_1$  be the largest element in $P'$ whose second coordinate is $j_1$.
We include to $C_1$ all the elements of $P'$ between $(i_1,j_1)$ and $(i_2,j_1)$.
Now let $j_2>j_1$ be the minimal element such that $P'$ contains an element whose 
second coordinate is $j_2$ and the first coordinate is no smaller than $i_2$.
Let $(i_3,j_2)$ be the smallest element with this property and let $(i_4,j_2)$
with $i_4\ge i_3$ be the largest element in $P'$ whose second coordinate is $j_2$
and the first coordinate is no smaller than $i_3$. 
We include to $C_1$ all the elements of $P'$ between $(i_3,j_2)$ and $(i_4,j_2)$.
We proceed further until it is possible, i.e. until we are able to find $j_{\bullet+1}>j_\bullet$ with the properties as above. 
By construction, $C_1$ is indeed a chain and for any $(i,j)$ such that there exists an element
$(i',j')\in C_1$ with $(j\le j')$, $i>i'$ one has $(i,j)\notin P'$.   
After $C_1$ is fixed, the chain $C_2$ is constructed using the same algorithm for the subposet $P'\setminus C_1$. We proceed until we cover the whole set $P'$ with the chains $C_1,\dots,C_r$.
We illustrate the procedure by the following picture. 
The elements of $P$ are represented by circles (the element $(i,j)$ is located in 
the $i$-th row and $j$-th column). 
The elements of $P'$ are filled circles and the chains $C_i$ are marked by segments.

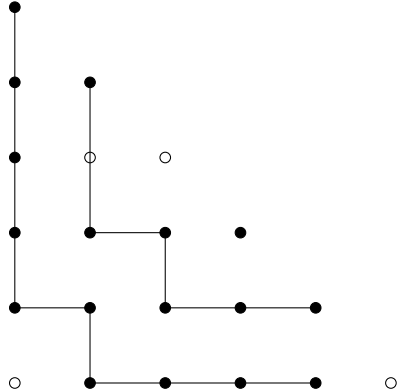
\[
\begin{tikzpicture}
	\draw (0,0) circle (2pt);
	\filldraw (1,0) circle (2pt);
	\filldraw (2,0) circle (2pt);
	\filldraw (3,0) circle (2pt);
	\filldraw (4,0) circle (2pt);
	\draw (5,0) circle (2pt);
	
	\filldraw (0,1) circle (2pt);
	\filldraw (1,1) circle (2pt);
	\filldraw (2,1) circle (2pt);
	\filldraw (3,1) circle (2pt);
	\filldraw (4,1) circle (2pt);
	
	\filldraw (0,2) circle (2pt);
	\filldraw (1,2) circle (2pt);
	\filldraw (2,2) circle (2pt);
	\filldraw (3,2) circle (2pt);
	
	\filldraw (0,3) circle (2pt);
	\draw (1,3) circle (2pt);
	\draw (2,3) circle (2pt);
	
	\filldraw (0,4) circle (2pt);
	\filldraw (1,4) circle (2pt);
	
	\filldraw (0,5) circle (2pt);
	
	\draw (0,5) -- (0,4) -- (0,3) -- (0,2) -- (0,1) -- (1,1) -- (1,0) -- (2,0) -- 
	(3,0) -- (4,0);
	
	\draw (1,4) -- (1,3) -- (1,2) -- (2,2) -- (2,1) -- (3,1) -- (4,1);
	
\end{tikzpicture}
\]

One easily sees that
\[
\sum_{a=1}^r\sum_{(i,j)\in C_a} x_{i,j} = \sum_{(i,j)\in P'} x_{i,j}.
\]
We claim that
\[
\sum_{a=1}^r (m_{s(C_a)}+\dots +m_{f(C_a)-1}) = \sum_{d=1}^{n-1} m_d w(P'\cap P_d).
\]
In fact, the left hand side is no smaller than the right hand side because the poset
$P'\cap P_d$ can not be covered by less than $w(P'\cap P_d)$ chains; however, any
chain $C_a$ intersecting $P'\cap P_d$ adds a summand $m_d$ to the left hand side.
In the opposite direction, let $C_{r_1},\dots,C_{r_w}$,  $1\le r_1 < \dots < r_w \le r$
be chains that intersect nontrivially with $P'\cap P_d$. We show that the width
of the intersection of $P'\cap P_d$ with $\cup_{u=1}^w C_{r_u}$ is exactly $w$. The proof
goes by induction on $w$ (note that the case $w=1$ is trivial).

Let us consider  $(P'\cap P_d)\setminus C_{r_1}$. Let 
$(p_1,q_1),\dots, (p_y,q_y)$ be the maximal size set of pairwise non comparable 
elements of $(P'\cap P_d)\setminus C_{r_1}$. Without loss of generality we assume 
that $q_1<\dots<q_y$ and  $p_1>\dots > p_y$. Then there are two options.
The first option: there exists a point $(p,q)\in C_{r_1}\cap  (P'\cap P_d)$ such that
$p>p_1$ and $q<q_1$. Then the width of $(P'\cap P_d)\bigcap (\cup C_{r_\bullet})$
is equal to $y+1$ and by induction $y=w-1$. The second option: 
$C_{r_1}\cap  (P'\cap P_d)$ does not contain points $(p,q)$ such that $p>p_1$ and $q>q_1$. Since $t_1$ is smaller than the index of the chain passing through $(p_1,q_1)$
there exists $p<p_1$ such that $(p,q)\in C_{r_1}\cap (P'\cap P_d)$ for some $q$.
Since we are inside the second option, $q<q_1$. But then by construction of
our chains one gets $(p_1,q_1)\in C_{r_1}$, which is a contradiction.
\end{proof}

\subsection{Dual problem}
Now we are ready to prove the following theorem.

\begin{thm}\label{thm:MS}
For any $\bbm, \bbm'\in\bZ_{\ge 0}^{n-1}$ and $M,M'\ge 0$ the Minkowski sum of 
$S_{\bbm,M}$ and $S_{\bbm',M'}$ is equal to $S_{\bbm+\bbm',M+M'}$.
\end{thm}
\begin{proof}
One easily sees that $S_{\bbm,M}+ S_{\bbm',M'}\subset S_{\bbm+\bbm',M+M'}$.
Since 	$S_{\bbm,0}+S_{\bbm',0} = S_{\bbm+\bbm',0}$ and $S_{0,M}+S_{0,M'} = S_{0,M+M'}$,
it suffices to show that
$S_{\bbm,M}+ S_{0,1}= S_{\bbm,M+1}$. Following \cite{F5} we use the duality of linear
programming (see \cite{Sch}).

For $x=(x_{i,j})\in\bZ_{\ge 0}^P$ let 
\[
M(x) = \max_{P'\subset P} \left(\sum_{(i,j)\in P'} x_{i,j}\ - \sum_{d=1}^{n-1} m_d w(P'\cap P_d)\right).
\]
We note that $M(x)\ge 0$ (it is enough to take $P'=\emptyset$) and $x\in S_{\bbm,M}$
if and only if $M(x)\le M$. Our goal is to show that there exists an element $\delta\in P$
such that $M(x-\mathbf{1}_{\delta})\le M(x)-1$. To this end we reformulate the problem as follows.

By Lemma \ref{lem:manychains} the definition of $M(x)$ can be restated in terms of
collections of chains $C_1,\dots,C_r$, which we assume to be pairwise non-intersecting.
To such a set of chains we associate a collection of numbers $k_{\alpha,\beta}\in\{0,1\}$, $\al,\beta\in \overline{P}$, $\al<\beta$ in the following way:
\begin{itemize}
\item $k_{\al,\beta}=1$ for $\al,\beta\in P$ if and only if $\al$ and $\beta$ are
consecutive elements in one chain;
\item  for $\al\in P$, $\beta\notin P$ the value  $k_{\al,\beta}$ is one if and only if $\al$ is the maximal element of a chain $C_a$ and $f(C_a)=\beta$,
\item  for $\al\notin P$, $\beta\in P$ the value  $k_{\al,\beta}$ is one if and only if $\beta$ is the minimal element of a chain $C_a$ and $s(C_a)=\alpha$,
\item if $\al,\beta\notin P$, then $k_{\al,\beta}=0$.
\end{itemize}
Clearly, one gets all the collections $(k_{\al,\beta})\in\{0,1\}$ subject to the conditions: for all $\beta\in P$
\begin{equation}\label{eq:k=k<1}
\sum_{\substack{\al\in \overline{P}\\ \al<\beta}} k_{\al,\beta} =
\sum_{\substack{\gamma\in \overline{P}\\ \gamma>\beta}} k_{\beta,\gamma} \le 1.
\end{equation}
In terms of $k_{\al,\beta}$ the expression $M(x)$ to be maximized reads as
\[
\sum_{\substack{\al\in\overline{P},\beta\in P\\ \al<\beta}} x_\beta k_{\al,\beta} - 
\sum_{\substack{\al\in{P}, i\in [n]\\ \al<(i,i)}} (m_1+\dots +m_{i-1})k_{\al,(i,i)} +
\sum_{\substack{\beta\in{P}, i\in [n]\\(i,i)<\beta}} (m_1+\dots +m_{i-1})k_{(i,i),\beta}.
\]
We note that conditions \eqref{eq:k=k<1} are the same as the ones used in \cite{F5}. Hence 
the dual linear program also involves variables $g_\al$ and $h_\beta$ and 
the expression to minimize is the sum of all variables $g_\al$.
Now the same argument as in the proof of Theorem A.1 in \cite{F5} shows that for the
dual linear program one finds an element $\delta\in P$ such that  
$M(x-\mathbf{1}_\delta)$ is less than $M(x)$, where $M(x)$ is the (common) minimal value
for the dual program and maximal value for the initial program.
\end{proof}

\begin{cor}
For any $\bbm, \bbm'\in\bZ_{\ge 0}^{n-1}$ and $M,M'\ge 0$ one has 
$X_{\bbm,M}+X_{\bbm',M'}=X_{\bbm+\bbm',M+M'}$.
\end{cor}
\begin{proof}
The proof is standard: using the claim for integer points one derives the analogues statement for rationals, which implies the general case.
\end{proof}

\section{Representation theory}\label{sec:repth}
Recall the Cartan decomposition $\msl_n=\fn_-\oplus\fb$, $\fb=\fh\oplus\fn$ and let
us denote by $\Phi=\Phi_+\sqcup \Phi_-$ the set of roots of $\msl_n$.
In particular, the set of positive roots $\Phi_+$ consists of elements
$\al_{i,j}$, $1\le i\le j\le n-1$ with $\al_i=\al_{i,i}$ being the set of
simple roots. One has $\al_{i,j}=\al_i+\dots +\al_{j}$. For a positive
root $\al_{i,j}$ we denote by $f_{i,j}=f_{\al_{i,j}}\in\fn_-$ and $e_{i,j}=e_{\al_{i,j}}\in\fn$ the corresponding 
root vectors. In particular, $f_{i,j}=E_{j+1,i}$ (with $E_{\bullet,\bullet}$ 
being matrix units).   

Let $\la=\sum_{d=1}^{n-1} m_d\om_d$ be an integral dominant weight of the Lie algebra $\msl_n$ (i.e. $m_d\in\bZ_{\ge 0}$). Let $L_\la$ be the irreducible highest weight $\msl_n$-module of highest weight $\la$. Let $\ell_\la\in L_\la$ 
be the highest weight vector. The universal enveloping algebra $\U(\fn_-)$
generates $L_\la$ from $\ell_\la$, i.e. $L_\la=\U(\fn_-)\ell_\la$.

The standard PBW filtration on the universal enveloping algebra induces
the increasing filtration on $L_\la$:
\[
L_\la(s)=\mathrm{span}\{a_1\dots a_r \ell_\la: a_i\in\fn_-, r\le s\}.
\]
The associated graded space is denoted by $L_\la^a$, which is a module over the 
degenerate Lie algebra $\msl_n^a\simeq \fn^a_-\oplus \fb$, where $\fn^a_-$
is an abelian ideal and the subalgebra $\fb$ acts on $\fn^a_-$ as on the quotient
module $\fg/\fb$.  We denote the image of $\ell_\la$ inside $L_\la^a$ by the
same symbol. 

We consider the cyclic $\msl_n^a$ module $V\simeq V(0) \oplus V(1)$ with a cyclic vector $v$ spanning $V(0)$ and $V(1)\simeq \fn^a_-$. The subalgebra $\fb\subset\msl_n^a$ acts trivially on $V(0)$ and acts on $V(1)$ via the 
identification $V(1)\simeq \fg/\fb$. The action of $\fn_-^a\subset\msl_n^a$
is trivial on $V(1)$ and on $V(0)$ the action of $\fn^a_-$ comes from the
identification $V(0)\T \fn_-^a\simeq V(1)$.   

Given a non-negative integer $M$ we define 
\[
V_M=V^{\odot M} = \U(\fn^a_-)v^{\T M},\ 
L^a_{\la,M}= L_\la^a \odot V_M = \U(\fn_-^a) (\ell_\la\T v^{\T M}).
\]
In particular, $L^a_{\la,M}$ are $\msl_n^a$ modules and $L^a_{\la,0}\simeq L^a_\la$,
$L^a_{0,M}\simeq V_M$. 
We denote the cyclic vector of $L^a_{\la,M}$ by $\ell_{\la,M}$. 

\begin{lem}
	The defining relations of the cyclic $\msl_n^a$ module $V_M$ are of the form
	\[
	f_{\beta_1}\dots f_{\beta_p} v_M = 0,\ \beta_i\in\Phi_+, p > M. 
	\]
\end{lem}
\begin{proof}
	The module $V\simeq V_1$ is isomorphic to the quotient of the polynomial
	ring in variables $f_\beta$ by the ideal generated by all the quadratic expressions. Hence $V_M$ is isomorphic to the quotient of the 
	same polynomial ring by the ideal generated by all the expressions of degree
	$M+1$.
\end{proof}

Recall the standard scalar product on $\fh^*$ defined by $(\al_i,\omega_j)=\delta_{i,j}$.

\begin{lem}\label{lem:rel}
	The following relations hold true in $L^a_{\la,M}$: for any $r\ge 1$ and any
	$\beta_1,\dots,\beta_r\in\Phi_+$:
	\begin{equation}\label{eq:defrel}
		f_{\beta_1}^{a_1}\dots f_{\beta_r}^{a_r} \ell_{\la,M} = 0 \text{ if }
		a_1+\dots + a_r > M + \sum_{i=1}^r (\la,\beta_i).
	\end{equation}
\end{lem}
\begin{proof}
	Recall that  for any $\beta\in \Phi_+$ the following relation holds true
	in $L_\la^a$: $f_\beta^{(\la,\beta)+1}\ell_\la =0$ (this relation holds true even 
	before passing to the associated graded space, i.e. already in $L_\la$).
	Recall that $L^a_{\la,M}=\U(\fn_-^a)(\ell_\la\T v_M)$ ($v_M$ is a cyclic vector
	of $V_M$). Since $f_{\beta_i}^{(\la,\beta_i)+1}\ell_\la =0$, the expression
	$f_{\beta_1}^{a_1}\dots f_{\beta_r}^{a_r} (\ell_\la\T v_M)$ is equal to a linear
	span of vectors of the form 
	\[
	\prod_{i=1}^r f_{\beta_i}^{b_i}\ell_\la\T \prod_{i=1}^r f_{\beta_i}^{c_i} v_M,\
	b_i+c_i=a_i, b_i\le (\la,\beta_i). 
	\] 
	Hence each summand as above vanishes, because $\sum_{i=1}^r c_i>M$. 
\end{proof}

Recall the polytopes $X_{\bbm,M}$ and the sets of integer points $S_{\bbm,M}\subset X_{\bbm,M}$. In what follows we write $X_{\la,M}$ and $S_{\la,M}$ for $\la=\sum_{i=1}^{n-1} m_i\om_i$. We also relabel the coordinates 
$x_{i,j}$, $n\ge i>j\ge 1$ by $x_{\al_{j,i-1}}$ (recall that $f_{\al_{j,i-1}}$
is equal to the matrix unit $E_{i,j}$). To $\bs=(s_\al)_{\al\in\Phi_+}\in S_{\la,M}$ we attach the vector $f^\bs\ell_{\al,M}=\prod_{\al} f_\la^{s_\al}\ell_{\la,M}$.  

\begin{thm}\label{thm:rel+bas}
	Relations \eqref{eq:defrel} are defining for the $\msl_n^a$ module $L^a_{\la,M}$. The vectors  $f^\bs\ell_{\al,M}$, $\bs\in S_{\la,M}$ form
	a basis of $L_{\la,M}$.
\end{thm}
\begin{proof}
	The proof is analogous to the proof of Theorem 2.19, \cite{F5}.
	Linear independence of vectors   $f^\bs\ell_{\al,M}$, $\bs\in S_{\la,M}$
	is implied by Theorem \ref{thm:MS} and \cite{FFL2}. The spanning property
	is proved in the following way. We need to show that if $\bt\notin S_{\la,M}$, 
	then $f^\bt\ell_{\al,M}$ can be expressed as a linear combination of vectors
	of the form $f^\bs\ell_{\al,M}$, $\bs\in S_{\la,M}$. This is achieved as in the
	proof of Proposition 2.18, \cite{F5} using Lemma \ref{lem:rel} and Lemma \ref{lem:manychains}. 
\end{proof}

\section{Geometry}\label{sec:geom}
Let $L$ be an $n$-dimensional vector space with a fixed basis $\ell_1,\dots,\ell_n$.
All the Grassmannians $\Gr(d,n)$ considered below are assumed to consist of subspaces 
of $L$ (of dimension $d$).
We denote by $\mathrm{pr}_i: L\to L$ the projection along the $i$-th basis vector, i.e.
$\mathrm{pr}_i \ell_j= \delta_{i,j}\ell_j$. In what follows we denote by $E_{i,j}\in \mgl(L)\simeq \mgl_n$ the matrix units with respect to the basis $\ell_i$.

\subsection{Graph closures}
We recall the setup of \cite{BSS,F5}. Let $\imath:\bA^{d(n-d)}\to\Gr(d,n)$ be the 
standard parametrization of the open affine cell in the Grassmannain.
Recall that this open cell consists of the subspaces being the column spans of the following
matrices
\[
\begin{pmatrix}
	1 & 0 &  \ldots &  0\\
	0 & 1 &  \ldots &  0\\
	\ldots & \ldots & \ldots & \ldots\\
	0 & 0 &  \ldots & 1\\
	z_{d+1,1} & z_{d+1,2} &  \ldots & z_{d+1,d}\\
	z_{d+2,1} & z_{d+2,2} &  \ldots & z_{d+2,d}\\
	\ldots & \ldots & \ldots & \ldots\\
	z_{n,1} & z_{n,2} &  \ldots & z_{n,d} 	
\end{pmatrix}.	
\]
The map $\imath$ extends to the birational map $\bP^{d(n-d)}\to\Gr(d,n)$
(both the projective space and the Grassmannian are compactifications of the affine space
$\bA^{d(n-d)})$ and we consider the closure of the graph of this birational map.
More precisely, we consider the closure $\G(d,n)\subset \bP^{d(n-d)}\times\Gr(d,n)$
of the points of the form $(x,\imath(x))$, $x\in\bA^{d(n-d)}$. 
The variety $\G(d,n)$
has two natural projections $\varphi:\G(d,n)\to\Gr(d,n)$ and $\psi:\G(d,n)\to \bP^{d(n-d)}$.

Let $\fp_d\subset\msl_n$ be the standard $d$-th maximal parabolic subalgebra 
and let $\fr_d$ be the corresponding abelian radical (hence $\msl_n=\fp_d\oplus \fr_d$).
The group $\bG_a^{d(n-d)}=\exp(\fr_d)$ acts on $\G(d,n)$ with an open orbit, the action
extends the natural action on $\Gr(d,n)$. More precisely, we write $\bP^{d(n-d)}$ as 
$\bP(\bC v\oplus \fr_d)$ for an auxiliary vector $v$ . Then the space $\bC v\oplus \fr_d$ can be seen as a $\fr_d$ module
with the trivial action on $\fr_d$ and with the obvious action $\bC\T \fr_d \to\fr_d$.
The abelian unipotent group $\exp(\fr_d)$ acts on both
$\bP^{d(n-d)}$ and $\Gr(d,n)$. 

Let us identify the radical $\fr_d$ with the space of homomorphisms $\mathrm{Hom}(L^-_d,L^+_d)$, where
\[
L^-_d=\mathrm{span}\{\ell_1,\dots,\ell_d\},\quad L^+_d=\{\ell_{d+1},\dots,\ell_n\}.
\]
Let $\partial\G(d,n)\subset \G(d,n)$ be the complement to the open dense cell  
$\G^o(d,n)$ (open dense $\bG_a^{d(n-d)}$ orbit). The boundary $\partial\G(d,n)$ belongs to 
$\bP(\fr_d)\times\Gr(d,n)\subset \bP^{d(n-d)}\times\Gr(d,n)$ and  
consists of pairs $([f],U)$ such that
\begin{equation}
	\ker(f)\supset pr_{[d+1,n]}(U),\qquad \mathrm{Im}(f)\subset U,	
\end{equation}
where $pr_{[d+1,n]}:L\to L^-_d$ is the projection along $L_d^+$ and $[f]$ is the line (an element of the 
projective space) containing $f$. 
We conclude that
\begin{gather*}
	\varphi^{-1}(U)\simeq \bP(\mathrm{Hom}\left(L_d^-/pr_{[d+1,n]}(U),U\cap L_d^+\right),\\
	\psi^{-1}([f])\simeq \Gr(d-\dim\mathrm{Im}(f),\ker(f)\oplus L_d^+).
\end{gather*}

\begin{rem}
Note that 	$pr_{[d+1,n]}=pr_{d+1}\dots pr_n$.
\end{rem}

\subsection{Degenerate flag varieties}
Recall that $\Gr(d,n)$ is the parabolic flag variety corresponding to the maximal parabolic subgroup. We extend the construction of the graph closure
to the case of PBW degenerate flag varieties. 
The reason we consider the degenerate flags instead of the classical ones is that in 
the latter case there is no group acting on the graph closure with an open orbit, 
while in the former case the abelian unipotent group still does the job (even the whole degenerate group $SL_n^a$ is available).
Let us recall the main definitions.

Let $\cF_n\subset \prod_{d=1}^{n-1} \Gr(d,n)$ be the flag variety for the group $SL_n$  consisting of collections 
$(U_d)_{d=1}^{n-1}$ such that $U_d\subset U_{d+1}$. One has
$\dim \cF_n = n(n-1)/2$; in what follows we denote this dimension by $N$,
we note that $N=\dim\fn$.
The flag varieties are acted upon by the group $SL_n$ 
and  $\cF_n$ is identified with the quotient $SL_n/B$ for the Borel subgroup
$B$.

Let $\cF_n^a$ be the PBW degenerate flag variety. By definition, this variety also sits
inside the product of Grassmann varieties $\Gr(d,n)$, $1\le d\le n-1$ and consists
of collections of subspaces $(U_d)_{d=1}^{n-1}$ of the ambient space $L$ such that
\[
\pr_{d+1} U_d\subset U_{d+1} \text{ for all } d=1,\dots,n-2. 
\]   
The varieties  $\cF_n^a$ are flat degenerations of the classical flag varieties  $\cF_n$; in particular, $\dim \cF_n^a =  N$.
The variety  $\cF_n^a$ admits an action of the  abelian unipotent group $\bG_a^{N}$. 
The group $\bG_a^{N}$ acts on $\cF_n^a$ with an open dense orbit 
isomorphic to the affine cell (we note that $\cF_n^a$ is irreducible).
However, there is a larger group acting on the degenerate flag varieties.  

Namely, recall the degenerate Lie algebra $\msl_n^a$ being 
the direct sum of its abelian ideal $\fn_-^a$ and its Lie subalgebra $\fb$.
Let us denote by $SL_n^a$ the following Lie group of the Lie algebra $\msl_n^a$:
the 
group $SL_n^a$ is the semi-direct product of its normal abelian subgroup $\bG_a^{N}=\exp(\fn^a_-)$ and 
the Borel subgroup $B$ (the defining homomorphism 
$B\to \mathrm{Aut}(\bG_a^N)$ is defined by the action of $\fb$ on $\fn^a_-$).

\subsection{The graph closure}
Let us generalize the definition of $\G(d,n)$ to the flag varieties case.
Let $V$ be a one-dimensional extension of the space $\fn_-^a$,
i.e.  $V=V(0)\oplus V(1)$, where $V(0)$ is a one-dimensional vector space spanned by a vector $v$ and $V(1)\simeq \fn^a_-$ (hence $\dim V=N+1$).
Then $V$ admits a natural structure of $\fn^a_-$ module with the trivial action on $V(1)$ and an obvious action $V(0)\T \fn^a_-\to V(1)$. 
We note that $V$ is cyclic $\fn^a_-$ module (with cyclic vector $v$) and 
$V$ is isomorphic to the quotient of the symmetric algebra $S(\fn^a_-)$
by the ideal generated by all quadratic expressions.
For a positive integer $M$ we also consider $V_M=S(\fn^a_-)/I_{M+1}$, where
the ideal $I_{M+1}$ is generated by all degree $M+1$ polynomials.
In particular, $V=V_1$.

Recall that the degenerate flag variety admits an open cell which is 
an orbit of the abelian unipotent group $\bG_a^N=\exp(\fn^a_-)$; this cell is naturally identified with the vector space $\bA^N\simeq \fn^a_-$.
Let 
$\imath: \bA^{N}\subset \Fl^a_n$ be the corresponding parametrization.
The affine space $\bA^N\simeq \fn^a_-$ has a natural compactification $\bP(V)$ (since
$V$ is a one-dimensional extension of $\fn^a_-$).  
We define 
\[
\G(n) = \overline{\{(x,\imath(x)),\ x\in \bA^{N}\}},\qquad 
\G(n)\subset \bP(V)\times \Fl^a_n.
\]
Let $\varphi:\G(n)\to\Fl^a_n$ be the projection map.

\begin{lem}
	The variety $\G(n)$ admits an action of the group $SL_n^a$.
	The unipotent additive subgroup $\bG_a^N\subset SL_n^a$ acts on $\G(n)$ with an open dense orbit. 	
\end{lem}
\begin{proof}
	We note that $V$ has a natural structure of $SL^a_n$ module. More precisely,
	the action of the subgroup $\bG_a^N$ comes from the $\fn_-^a$ action and $B$ acts trivially on $V(0)$ and via the adjoint action on $V(1)\simeq \msl_n/\fb$. The diagonal part  $\{(x,\imath(x)),\ x\in \fn_-^a\}$ is by definition  an $\exp(\fn_-^a)$ orbit, which implies the claim of the lemma.  
\end{proof}

\begin{rem}
	Recall that the degenerate flag variety $\Fl^a_n$ is realized inside the product
	of Grassmannians $\Gr(d,n)$. However, the graph closure $\G(n)$ does not
	sit inside the graph closures
	$\G(d,n)$, $d=1,\dots,n-1$. The reason is as follows. Recall that 
	$\G(d,n)$ sits inside $\bP^{d(n-d)}\times \Gr(d,n)$, where $\bP^{d(n-d)}$ is the projective space of the radical of the $d$-th maximal parabolic subalgebra extended by a one-dimensional vector space.
	By definition, $\G(n)$ belongs to $\bP^N\times \Fl_n^a$. However, 
	the $\bG^n_a$ orbit passing through the product of cyclic lines of
	$\bP^{d(n-d)}$ is not isomorphic to $\bP^N$ 
	(already for $n=3$).
\end{rem}

Let $\partial\G(n)$ be the boundary of the graph closure, i.e. the complement to the
open cell $\G^o(n)=\{(x,\imath(x)),\ x\in \bA^{N}\}$. Our goal is 
to derive an explicit description of $\partial\G(n)$.
We start with the following simple observation.

\begin{lem}
	The boundary $\partial\G(n)$ belongs to $\bP(V(1))\times \Fl^a_n$.
\end{lem}
\begin{proof}
	It suffices to recall that $V\simeq V(0)\oplus V(1)$ with $V(1)\simeq \fn_-^a$ and 
	the Lie algebra $\fn^a_-$ acts from the one-dimensional $V(0)$ to $V(1)$.
\end{proof}

Let ${\bf J}=(J_1,\dots,J_{n-1})$ be a collection of subsets of the set $[n]$ such
that $\# J_d=d$ for all $d$ and  
\[
J_d\subset J_{d+1}\cup\{d+1\},\quad d=1,\dots,n-2.
\]
In what follows we call such collections admissible.
Recall that admissible collections ${\bf J}$ label the points of the degenerate 
flag variety  with each component being a coordinate subspace, i.e.
the torus fixed points of $\Fl^a_n$. In particular, the number
of such collections is equal to the normalized median Genocchi number \cite{F1,F2}.
Let $U_{\bf J}\in \Fl^a_n$ be the corresponding point, i.e. 
$U_{{\bf J},d}=\mathrm{span}\{\ell_i, i\in J_d\}$.

For $d\in [n-1]$ recall the radical $\fr_d$ of the $d$-th maximal
parabolic subalgebra (i.e $\fr_d$ is spanned by matrix units $E_{i,j}$
with $i>d$ and $j\le d$). These radicals are abelian and admit natural emebeddings into $\fn_-^a$. We also have projections $\kappa_d: \fn_{-}^a\to \fr_d$
along matrix units that are not in $\fr_d$. 
For an element $f\in\fn_-^a$ let $f_d=\kappa_d(f)$.
In words, we take a lower triangular matrix and keep only the entries
located in the first $d$ columns and the last $n-d$ rows.

Recall the projection $\varphi: \G(n)\to \Fl^a_n$.
\begin{prop}\label{prop:fiberJ}
	The fiber $\varphi^{-1}(U_{\bf J})$ is a projectivized subspace of 
	$\fn_-^a$ spanned by matrix units $E_{i,j}$ such that for any $d\in [n-1]$
	the following condition holds:
	\[
	\text{ if }\quad j\le d <i,\quad \text{ then }\quad i\in J_d, j\notin J_d.   
	\] 
\end{prop}	
\begin{proof}
	We first show that if $([f],U_{\bf J})\in \G(n)$ and $f=\sum_{i>j}f_{i,j}E_{i,j}$, then
	$f_{i,j}=0$ unless for any $d$ such that $j\le d <i$ one has $i\in J_d, j\notin J_d$.
	Recall that in \cite{F5}, Proposition 3.11 a similar statement was proved in the Grassmann case. More precisely, it was shown that the Pl\"ucker type relations
	cutting the graph closure for the Grassmannians inside the product of two
	projective spaces $\bP^{d(n-d)}\times \bP(\Lambda^d(L))$ imply that if a point 
	$([f_d],U_{J_d})$ belongs to $\G(d,n)$, then $(f_d)_{i,j}=0$ (with $j\le d <i$) unless
	$i\in J_d, j\notin J_d$. We have a natural embedding
	\[
	\partial \G(n)\subset \bP(\fn_-^a)\times\prod_{d=1}^{n-1} \bP(\Lambda^d(L)) 
	\]
	and the image is cut out by certain ideal in the polynomial ring in variables $Y_{i,j}$,
	$n\ge i>j\ge 1$ (responsible for the projective space $\bP(\fn_-^a)$)
	and $X_{i_1,\dots,i_d}$ with $d\in[n-1]$ and $1\le i_1<\dots <i_d\le n$
(responsible for the Grassmannian).	
	Since all the relations coming from the Grassmann case are present in our situation 
	as well, we conclude that $f_{i,j}$ vanishes unless the conditions above hold true.
	
	Now assume that we are given an element $f\in\fn_-^a$ such that
$f_{i,j}=0$ unless $(i\in J_d, j\notin J_d)$ for any $j\le d <i$.
	Let us consider the point
	\[
	p(t)= \left([v+tf],\prod_{d=1}^{n-1} \mathrm{span}\{\ell_a, a\in J_d\cap [d]\}\oplus 
	\mathrm{span}\{\ell_a + tf_d\ell_a, a\in [d]\setminus J_d\} \right).
	\]
By definition, $p(t)\in\G^o(n)$ -- the open cell of $\G(n)$. 	
	Assume that for any $d$ the operator 
	\begin{equation}\label{eq:fd}
		\mathrm{span}\{\ell_a, a\in [d]\setminus J_d\}\to \mathrm{span}\{\ell_a, a\in J_d\cap [d+1,n]\}
	\end{equation}
obtained by taking the corresponding entries of $f$
	is an isomorphism (note that the dimensions of the left and right hand sides do coincide). Then 
	$$\lim_{t\to\infty} \mathrm{span}\{\ell_a, a\in J_d\cap [d]\}\oplus 
	\mathrm{span}\{\ell_a + tf_d\ell_a, a\in [d]\setminus J_d\} = U_{J_d};$$
	since $\lim_{t\to\infty} [v+tf] = [f]$ we conclude that $\lim_{t\to\infty} p(t)=([f],U_{\bf J})$. Now the conditions that \eqref{eq:fd} are isomorphisms 
	cut out an open subset in $\fn_-^a$. However, the preimage $\varphi^{-1}(U_{\bf J})$
	is identified with a closed subvariety in $\bP(\fn^a_-)$. Hence 
	$\varphi^{-1}(U_{\bf J})$ is the desired projectivized subspace of $\fn_-^a$.   
\end{proof}

\begin{cor}\label{cor:ij}
The preimage $\varphi^{-1}(U_{\bf J})$ is the projectivization of the span of $E_{i,j}$ such that $i\in J_j$, $j\notin J_{i-1}$.
\end{cor}
\begin{proof}
First, $i\in J_j$, $i>j$ implies $i\in J_d$ for all $d$ such that $j\le d <i$
(recall that $J_d \subset J_{d+1}\cup\{d+1\}$). Second, $j\notin J_{i-1}$, $i>j$
implies $i\notin J_d$ for  	all $d$ such that $j\le d <i$. Now Proposition
\ref{prop:fiberJ} provides the desired claim.
\end{proof}

Recall that for a matrix $f\in\fn^a_-$ the operator $f_d: L_d^-\to L_d^+$ is obtained
by keeping only the entries of $f$ located in the first $d$ columns and last $n-d$ rows. 

\begin{thm}\label{thm:maingeom}
	For $U=(U_d)_d\in\partial \Fl_n^a$ the fiber $\varphi^{-1}(U)$ consists of pairs
	$([f],U)$ such that for any $d\in [n-1]$ one has
	\[
	\mathrm{Im}(f_d)\subset U_d,\quad \ker(f_d)\supset pr_{[d+1,n]} U_d. 
	\]
\end{thm}
\begin{proof}
	Thanks to Proposition 	\ref{prop:fiberJ} our theorem holds true
	for $U=U_{\bf J}$ for an admissible ${\bf J}$. We derive the general case using
	the action of the degenerate group $SL_n^a$.
	
	The group $SL_n^a$ acts on $\G(n)$ and on $\Fl^a_n$ and the map $\varphi$ is 
	$SL_n^a$ equivariant. The $SL_n^a$ orbits of the points $U_{\bf J}$ form an affine 
	paving of  $\Fl^a_n$. Hence it suffices to show that $SL_n^a.\varphi^{-1}(U_{\bf J})$
	agrees with the description given in the statement of our Theorem.
	More precisely, we need to show that if 
	$\mathrm{Im}(f_d)\subset U_d$ and $\ker(f_d)\supset pr_{[d+1,n]} U_d$, then
	for any $g\in SL_n^a$ one has
	$\mathrm{Im} (gf)_d\subset gU_d$ and $\ker (gf)_d\supset pr_{[d+1,n]} gU_d$.
	
	Let us write $g=\gamma\exp(\tau)$ for $\gamma\in B$
	and $\tau\in\fn_-^a$. 
	The action of $\exp(\tau)$ on $f$ is trivial. We also have 
\[
U_d\cap L_d^+ = \exp(\tau)U_d\cap L_d^+,\   pr_{[d+1,n]} \exp(\tau)U_d = pr_{[d+1,n]} U_d.
\]
	Hence we can assume that $\tau=0$, i.e. $g=\gamma\in B$.
	Now it suffices to note that, on the one hand,  the $B$ action on $U_d$ comes from the actions 
	on $L_d^-$ (which is a $B$ submodule of $L$) and on $L_d^+$ (which is a 
	quotient $B$-mdoule $L/L_d^-)$, and, on the other hand, the action of $B$ 
	on the collection $f_d$  (which is the conjugation action) is also glued
	from two pieces (the action on the source of $f_d$ and on its target).
\end{proof}

\begin{rem}
	Theorem \ref{thm:maingeom} says that the fibers are cut out by the combination of conditions which show up in the Grassmann case (see \cite{F5}).
\end{rem}

\subsection{Quiver Grassmannians}
The goal of this subsection is to restate Theorem \ref{thm:maingeom} in terms of 
representations of quivers.

Let $Q$ be the equioriented quiver of type $A_{n-1}$. For a pair of numbers 
$1\le a\le b\le n-1$ we denote by $M_{a,b}$ the indecomposable $Q$-module
with one-dimensional spaces at each vertex from $a$ to $b$. In particular,
$M_{1,b}=I_b$ are injective modules and $M_{a,n-1}=P_a$ are projective modules
(note that $I_{n-1}=P_1=M_{1,n-1}$). Let $P=\bigoplus_{a=1}^{n-1} P_a$ and let
$I=\bigoplus_{b=1}^{n-1} I_b$ be the direct sums of all indecomposable projective and all indecomposable injective representations. We note that
$\dim P = (1,2,\dots,n-1)$ and $P$ is isoomorphic to the path algebra of $Q$.
Also, $\dim I=(n-1,\dots,1)$ and as a $Q$-module $I$ is the dual path algebra.
It was observed in \cite{CFR1} that the degenerate flag variety $\Fl^a_{n}$
is isomorphic to the quiver Grassmannian $\Gr_{\dim P} (P\oplus I)$ (this observation has several interesting applications, see e.g. \cite{CFR2,CFR3,CL,FeFi}).   

Recall that a point in the quiver Grassmannian $\Gr_{\dim P} (P\oplus I)$
is a subrepresentation $N\subset P\oplus I$ of dimension $\dim P$.
It was shown in \cite{CFR1} that $N\simeq N_P\oplus N_I$ for certain subrepresentations $N_P\subset P$, $N_I\subset I$ (in particular, $N_I\simeq N\cap I$).  So from now on we identify $U=(U_d)_d\in \Fl_n^a$ with a point
$N_P\oplus N_I$ from the quiver Grassmannian. 

\begin{rem}
	It was shown in \cite{CFR2} that $U$ is a smooth point if and only if 
	$\mathrm{Ext}^1(N_I,P/N_P)$ vanishes.
\end{rem}

\begin{prop}\label{prop:Hom}
	Let ${\bf J}=(J_d)_{d=1}^{n-1}$ be an admissible collection and let 
	$U_{\bf J}\in \Fl_n^a$ be the corresponding point. Then for a $Q$-module
	$N_P\oplus N_I$ representing the point $U_{\bf J}$ one has   
	$\varphi^{-1} U_{\bf J} \simeq \bP\mathrm{Hom}_Q(P/N_P,N_I)$.
\end{prop}
\begin{proof}
By Corollary \ref{cor:ij} we need to show that for $N=U_{\bf J}$ one has 
	\begin{equation}\label{eq:ij}
		\mathrm{Hom}_Q(P/N_P,N_I) \simeq \mathrm{span}\{E_{i,j}: i > j, i\in J_j, j\notin J_{i-1}\}.
	\end{equation}
	Since as $Q$-modules $U_{\bf J}\simeq N_P\oplus N_I$, $N_P\subset P$, $N_I\subset I$, one has:
	\begin{gather*}
		P/N_P\simeq \bigoplus_{a=1}^{n-1} M_{a,t(a)},\quad
		N_I = \bigoplus_{b=1}^{n-1} M_{s(b),b}, 
	\end{gather*}
	where $t(a)$ is defined by $a\notin J_{t(a)}$, $a\in J_{t(a)+1}$ and
	$s(b)$ is defined by $b+1\in J_{s(b)}$, $b+1\notin J_{s(b)-1}$.
	The hom-space $\mathrm{Hom}_Q(M_{a,t(a)},M_{s(b),b})$ is one-dimensional if 
	$s(b)\le a\le b\le t(a)$ and vanishes otherwise. We note that 
	$s(b)\le a$ is equivalent to $b+1\in J_a$ and $b\le t(a)$ is equivalent to 
	$a\notin J_b$. Now redefining $i=b+1$, $j=a$, we arrive at \eqref{eq:ij}.
\end{proof}

\begin{lem}
	For an admissible collection ${\bf J}$ let $N_P\oplus N_I$ be the 
	quiver representation corresponding to the point $U_{\bf J}$. Then
	$\mathrm{Hom}_Q(P/N_P,N_I)$ vanishes if and only if ${J}_d=[d]$ for all $d$.  
\end{lem}
\begin{proof}
Recall that the set of admissible collections ${\bf J}$ form a poset with 
the poset structure induced by the orbits closure. The largest element in the
poset is given by ${\bf J}= ([d])_{d=1}^{n-1}$ and any other element can be reached from the
	largest one by applying certain mutation operations (this can be seen explicitly
	from the description of the moment graph \cite{CFR2} or via the identification
	of the degenerate flag varieties with certain type $A$ Schubert varieties \cite{CL}). 
	The mutations are labeled by pairs $a>b$ and the action is defined as
	follows: the mutation $a,b$ affects only $J_d$ with $b\le d<a$. If these 
	inequalities hold true, then $J_d$ is affected if and only if $b\in J_d$,
	$a\notin J_d$. In this case $J_d$ is changed to $J_d\setminus \{b\}\cup \{a\}$.
	
	Now assume that  $\mathrm{Hom}_Q(P/N_P,N_I)=0$, i.e. there is no $i > j$ such that
	$i\in J_j$ and $j\notin J_{i-1}$ (see \eqref{eq:ij}). Our goal is to deduce that
	$J_d=[d]$. We start with the case $i=j+1$. Our condition says that 
	if $j+1\in J_j$, then $j\in J_j$ as well. This means that in the process 
	of getting ${\bf J}$ from the largest collection $([d])_d$ the first mutation applied was not of the form $(j+1,j)$. Now let us consider the case $i=j+2$. Then we know
	that $j+2\in J_j$ implies $j\in J_{j+1}$. It means that a mutation labeled by 
	$(j+2,j)$ was not applied first to the element $([d])_d$ (since otherwise 
	$j+2$ would be present in both $J_j$ and $J_{j+1}$ and $j$ would not be present
	in both sets). Proceeding in this way we obtain that no mutation could have
	been applied first when moving from the largest element to ${\bf J}$. Hence 
	$J_d=[d]$ for all $d$, $N_I=0$, $N_P=P$.
\end{proof}

\begin{thm}\label{thm:quiver}
	For any point ${U}\in \partial\Fl_n^a$ the preimage $\varphi^{-1}({U})$ is 
	isomorphic to $\bP\mathrm{Hom}_Q(P/N_P,N_I)$, where ${U}\simeq N_P\oplus N_I$ 
	as $Q$-modules.
\end{thm}
\begin{proof}
We fist note that the theorem is stated for a point $U$ from the boundary,
since otherwise the hom-space vanishes (on the open cell $\varphi^{-1}$ is an isomorphism). 
	
	The space $\mathrm{Hom}_Q(P/N_P,N_I)$ is identified with the subspace of 
	$\mathrm{Hom}_Q(P,I)$ consisting of homomorphisms vanishing on $N_P$ 
	whose image belongs to $N_I$. Recall the identification 
	$\mathrm{Hom}_Q(P,I)\simeq \fn^a_-$ (an indecomposable summand $M_{a,n-1}\subset P$
	can be mapped nontrivially to $M_{1,b}\subset I$ for $a\le b$). 
	This agrees with Theorem \ref{thm:maingeom}. 
\end{proof}

\subsection{Coordinate rings and degenerations}
In this subsection we collect geometric consequences from combinatorial and 
representation theoretic results.

\begin{cor}\label{cor:hcr}
	The homogeneous coordinate ring of $\G(n)$ with respect to the Pl\"ucker type
	embedding $\G(n)\subset \bP(V)\times\prod_{d=1}^{n-1} \bP(\Lambda^dL)$ is
	isomorphic to $\bigoplus (L^a_{\la,M})^*$, where 
	$\la\in \bigoplus_{i=1}^{n-1}\bZ_{\ge 0}\omega_i$, $M\ge 0$.
\end{cor}
\begin{proof}
The proof goes along the same lines as the proof of Corollary 3.18, \cite{F5}.
The key new ingredient is Theorem \ref{thm:rel+bas}. 
\end{proof}

\begin{cor}\label{cor:toricdeg}
	The variety $\G(n)$ admits flat degeneration to the toric variety defined 
	by the polytope $X_{\la,M}$ for any weight $\la$ with $m_d>0$ for any $d$ 
	and $M>0$.
\end{cor}
\begin{proof}
The proof goes along the same lines as the proof of Corollary 3.19, \cite{F5} using
Theorem \ref{thm:MS}.
\end{proof}

\section{Parabolic case}\label{sec:partial}
In this section we generalize the whole picture to the case of arbitrary 
(standard) parabolic subalgebra (the partial flag varieties case).
We note that \cite{F5} corresponds to the case of maximal parabolic subalgebra.

\subsection{The setup}
Let $\bd=(d_1,\dots,d_s)$ be a collection of positive integers satisfying 
$1\le d_1<\dots <d_s<n$; in particular, $s\le n-1$.
Let $\fp_\bd\subset\msl_n$ be the corresponding parabolic Lie subalgebra (i.e.
$\fp_\bd$ contains $\fb$ and negative roots $f_{\al_i}$ for $i=d_1,\dots,d_s$).
Let $\fr_\bd\subset \fn_-$ be the radical of $\fp_\bd$, $\msl_n=\fp_\bd\oplus \fr_\bd$. 

\begin{rem}
The radical $\fr_\bd$ is spanned by the matrix units $E_{i,j}$ such that there
exists an $a=1,\dots,s$ such that $d_{a-1}< j\le d_a$ and $i>d_a$ (with
$d_0=0$). One has $\dim \cF_\bd = \dim \fr_\bd=\sum_{a=1}^s (d_a-d_{a-1})(n-d_a)$.
In what follows we denote this number by $N_\bd$. 	
\end{rem}

\begin{rem}
	For $d=1,\dots,n-1$  recall the abelian radical $\fr_d\subset \fn_-$ corresponding
	to the $d$-th maximal parabolic subalgebra. Then $\fr_\bd=\sum_{a=1}^s \fr_{d_a}\subset \fn_-$.  	
\end{rem}

\subsection{Polytopes}
Let $P_\bd$ be the subposet of the poset $P$ (see Section \ref{sec:comb}) consisting of  
pairs $(i,j)$ such that $E_{i,j}\in \fr_\bd$. 
In particular, $P_\bd=\cup_{a=1}^s P_{d_a}$, where $P_{d_a}$ is the subposet
in $P$ corresponding to the maximal parabolic subalgebra (the $s=1$ case).  
The extension $P_\bd\subset \overline{P}_\bd\subset \overline{P}$ is obtained by adding elements $(1,1)$
and $(d_a+1,d_a+1)$ for all $a$. Now let $\bbm\in\bZ_{\ge 0}^{n-1}$ be a 
collection satisfying $m_i=0$ unless $i=d_a$ for some $a$.

For a non-negative integer $M$ we define a polytope $X_{\bd,\bbm,M}\subset \bR_{\ge 0}^{P_\bd}$ by the following set of inequalities
labeled by subposets $P'\subset P_\bd$:
\begin{equation}
	\sum_{\al\in P'} s_\al \le M + \sum_{a=1}^{s} m_{d_a} w(P'\cap P_{d_a}).
\end{equation} 
Let $S_{\bd,\bbm,M}=X_{\bd,\bbm,M}\cap \bZ_{\ge 0}^{P_\bd}$.
\begin{rem}
	The polytopes $X_{\bd,\bbm,M}$ do depend on $\bd$, because for different 
	$\bd$ the ambient space containing $X_{\bd,\bbm,M}$ is different. 
\end{rem} 

The following theorem is proved in the same way as Theorem \ref{thm:MS}.
\begin{thm}
	Let $\bbm,\bbm'\in\bZ_{\ge 0}$ satisfy $m_i=0$ unless $i=d_a$ for some $a$.
	Then for any $M,M'\ge 0$ one has
	\[
	S_{\bd, \bbm,M} + S_{\bd,\bbm',M'}=S_{\bd,\bbm+\bbm',M+M'},\ X_{\bd,\bbm,M} + X_{\bd,\bbm',M'}=X_{\bd,\bbm+\bbm',M+M'}.
	\]
\end{thm}

\subsection{Representations}
Let $\msl_{n,\bd}^a$ be the following degeneration of the Lie algebra $\msl_n$
(see \cite{F-M,PY2,Ya}).
One has a decomposition $\msl_{n,\bd}^a\simeq \fp_\bd\oplus \fr_\bd^a$, where 
$\fp_\bd$ is a subalgebra of $\msl_{n,\bd}^a$, $\fr_\bd^a$ is an abelian
ideal isomorphic to $\fr_\bd$ as a vector space and the action of $\fp_\bd$
on $\fr_\bd^a$ comes from the identification $\fr_\bd^a\simeq \msl_n/\fp_\bd$.

Let $\la=\sum_{a=1}^{s} m_{d_a}\om_{d_a}$ be an integral $\bd$-dominant weight 
of $\msl_n$. 
Then the corresponding irreducible highest weight $\msl_n$ module $L_\la$ 
is generated from the highest weight vector $\ell_\la$ by the action of the 
radical $\fr_\bd$. Hence the standard PBW filtration on the universal
enveloping algebra $\U(\fr_\bd)$ induces 
the increasing filtration on $L_\la$.
The associated graded space $L_\la^a$ is a module over the 
degenerate Lie algebra $\msl_{n,\bd}^a$.

Now let $V_\bd\simeq V_\bd(0) \oplus V_\bd(1)$ be a cyclic $\msl_{n,\bd}^a$
module, where  $V_\bd(0)$ is one-dimensional spanned by  a cyclic vector $v$
and  $V_\bd(1)$ is isomorphic to $\fr^a_\bd$ as a vector space. 
Given a non-negative integer $M$ and $\la$ as above we define 
\[
V_{\bd,M}=V_\bd^{\odot M} = \U(\fr_\bd^a)v^{\T M},\ 
L^a_{\bd,\la,M}= V_{\bd,M}\odot L_\la^a = \U(\fr_\bd^a) (\ell_\la\T v^{\T M}).
\]
As before, we denote the cyclic vector of $L^a_{\bd,\la,M}$ by $\ell_{\la,M}$. 

\begin{rem}
Let $\la$ be a $\bd$ dominant weight (i.e. $(\la,\al_i)=0$ unless $i\in\bd$). 
Then $L^a_{\la,M}$ is not isomorphic to  $L^a_{\bd,\la,M}$, since
$L^a_{0,1}\simeq \fn^a_-$ and $L^a_{\bd,0,1}\simeq \fr_\bd$.
\end{rem}

The following theorem is proved along the same lines as Theorem \ref{thm:rel+bas}.

\begin{thm}
	The relations 
	\[
	f_{\beta_1}^{a_1}\dots f_{\beta_r}^{a_r} \ell_{\la,M} = 0 \text{ if }
	f_{\beta_i}\in\fr_\bd, a_1+\dots + a_r > M + \sum_{i=1}^r (\la,\beta_i)
	\]
	are defining for the $\msl_{n,\bd}^a$ module $L^a_{\bd,\la,M}$. The vectors  $f^\bs\ell_{\la,M}$, $\bs\in S_{\bd,\la,M}$ form
	a basis of $L_{\bd,\la,M}$.
\end{thm}

\subsection{Partial flag varieties}
Let $\cF_\bd\subset \prod_{a=1}^s \Gr(d_a,n)$ be the partial flag variety consisting of collections 
$(U_{d_a})_{a=1}^s$ such that $U_{d_a}\subset U_{d_{a+1}}$. One has
\[
\dim \cF_\bd = \sum_{i=a}^s (d_a-d_{a-1})(n-d_a)=N_\bd.
\]
The partial flag varieties are acted upon by the group $SL_n$ 
and  $\cF_\bd$ is identified with the quotient $SL_n/P_\bd$ for the parabolic subgroup $P_\bd$. 

Let $\cF_\bd^a$ be the PBW degenerate flag variety. This variety also sits
inside the product of Grassmann varieties $\Gr(d_a,n)$, $1\le a\le s$ and consists
of collections of subspaces $(U_{d_a})_{a=1}^s$ of the ambient space $L$ 
such that
\[
\pr_{d_a+1}\dots\pr_{d_{a+1}} U_{d_a}\subset U_{d_{a+1}}. 
\]   
The varieties  $\cF_\bd^a$ are flat degenerations of the classical flag varieties  $\cF_\bd$;
in particular, $\dim \cF_\bd^a =  \dim \cF_\bd$.
The variety  $\cF_\bd^a$ admits an action of the  abelian unipotent group $\bG_a^{N_\bd}=\exp(\fr^a_\bd)$. 
The group $\bG_a^{N_\bd}$ acts on $\cF_\bd^a$ with an open dense orbit 
isomorphic to the affine cell (we note that the varieties $\cF_\bd^a$ are irreducible).
However, there is a larger group acting on the degenerate flag varieties.  
Let us denote by $SL_{n,\bd}^a$ a Lie group of the Lie algebra $\msl_{n,\bd}^a$
defined as the semi-direct product of its normal abelian subgroup $\bG_a^{N_\bd}=\exp(\fr^a_\bd)$ and 
the parabolic subgroup $P_\bd$ (the defining homomorphism 
$P_\bd\to \mathrm{Aut}(\bG_a^{N_\bd})$ is defined by the action of $\fp_\bd$ on $\fr^a_\bd$).

Let us generalize the definition of $\G(n)$ to the case of arbitrary collection $\bd=(d_1,\dots,d_s)$.
Recall the number $N_\bd=\dim\Fl_\bd=\dim\fr_\bd$ and the open cell $\imath: \bA^{N_d}\subset \Fl^a_\bd$.
We define 
\[
\G(\bd,n)\subset \bP^{N_\bd}\times \Fl^a_\bd,\qquad \G(\bd,n) = \overline{\{(x,\imath(x)),\ x\in \bA^{N_d}\}}.
\]
Let $\partial\G(\bd,n)$ be the boundary of the graph closure, i.e. the complement to the
open cell $\G^o(\bd,n)=\{(x,\imath(x)),\ x\in \bA^{N_d}\}$.
One has an embedding $\partial\G(\bd,n)\subset \bP(\fr_\bd)\times 
\Fl^a_\bd$.

Let $\varphi_\bd$ be the projection $\G(\bd,n)\to\Fl_\bd^a$.
For $f\in\fr_\bd$, $d\in\bd$ recall the element $f_d\in\fr_{d}$ 
(i.e. $f_d$ is the image of $f$ under the natural projection from $\fr_\bd$
to the radical corresponding to the maximal parabolic subalgebra).
The following theorem is proved along the same lines as Theorem \ref{thm:maingeom}.

\begin{thm}
	For $U=(U_d)_{d\in\bd}\in\partial \Fl_\bd^a$ the fiber $\varphi_\bd^{-1}(U)$ consists of pairs
	$([f],U)$ such that for any $d\in \bd$ one has $\mathrm{Im}(f_d)\subset U_d$ and $\ker(f_d)\supset pr_{[d+1,n]} U_d$. 
\end{thm}

Corollaries \ref{cor:hcr} and \ref{cor:toricdeg} hold for arbitrary $\bd$ as well.

\begin{cor}
The homogeneous coordinate ring of $\G(\bd,n)$ is isomorphic
to the direct sum of duals of $L_{\bd,\la,M}$. The variety  $\G(\bd,n)$ admits 
flat degeneration to the toric variety whose Newton polytope is $X(\bd,\bbm,M)$
(for $M>0$ and $m_{d_a}>0$, $a=1,\dots,s$).
\end{cor}

Finally, recall (see \cite{CFR1,CFR2}) that the partial degenerate flag varieties
also admit realizations as quiver Grassmmannians. More precisely, there exist
a projective and an injective representations $P_\bd$ and $I_\bd$ of 
the equioriented type $A_s$ quiver such that 
$\Fl^a_\bd\simeq \Gr_{\dim P_\bd} (P_\bd\oplus I_\bd)$. As in Theorem \ref{thm:quiver} we obtain that  a fiber of the projection $\varphi_\bd: \G(\bd,n) \to \Fl^a_\bd$ is the projectivization of the space $\mathrm{Hom}_{A_s}(P_\bd/N_P,N_I)$, where $N_P\oplus N_I$ represents a point in the degenerate flag variety or, equivalently, in the quiver Grassmannian.

\end{document}